\documentclass[amstex,a4paper,12pt,reqno]{amsart}
\usepackage{amssymb,amsmath,amsthm}
\usepackage[pdftex]{graphicx}
\usepackage{color}
\usepackage{appendix}

\usepackage{bm}
\usepackage{verbatim}

\def\R{{\mathbb R}}
\def\N{{\mathbb N}}

\definecolor{rojo}{rgb}{1,0,0}
\definecolor{azul}{rgb}{0,0,1}

\def\b{\beta}
\def\d{\delta}

\def\e{varepsilon}
\def\g{\gamma}
\def\s{\sigma}
\def\t{\theta}
\def\l{\lambda}
\def\p{\partial}
\def\O{\Omega}
\def\o{\omega}
\def\e{\varepsilon}
\def\v{\varphi}
\def\G{\Gamma}
\def\k{\kappa}

\def\mc{\mathcal}
\def\mf{\mathfrak}
\def\ua{\uparrow}

\numberwithin{equation}{section}

\theoremstyle{definition}

\theoremstyle{plain}
\newtheorem{theorem}{Theorem}[section]
\newtheorem{proposition}{Proposition}[section]

\theoremstyle{definition}

\begin{document}

\title[The weighted periodic-parabolic degenerate logistic equation]
{The weighted periodic-parabolic degenerate logistic equation}

\author{D. Aleja}
\address{Departamento de Matem\'{a}tica Aplicada, Ciencia e Ingenier\'{\i}a de los Materiales y Tecnolog\'{\i}a Electr\'{o}nica, Universidad Rey Juan Carlos, M\'ostoles, Spain}
\email{david.aleja@urjc.es}

\author{I. Ant\'on}
\address{Departamento de An\'alisis Matem\'atico y Matem\'atica Aplicada,
Universidad Complutense de Madrid,
Madrid 28040, Spain}
\email{iantonlo@ucm.es}

\author{J. L\'opez-G\'omez}
\address{Instituto de Matem\'atica Interdisciplinar (IMI),
Departamento de An\'alisis Matem\'atico y Matem\'atica Aplicada,
Universidad Complutense de Madrid,
Madrid 28040, Spain}
\email{Lopez\_Gomez@mat.ucm.es}

\maketitle

\begin{abstract}
The main goal of this paper is twofold. First, it characterizes the existence of positive periodic solutions for a general class of weighted periodic-parabolic logistic problems of degenerate type (see \eqref{1.1}). This result provides us with is a substantial generalization of Theorem 1.1 of Daners and L\'{o}pez-G\'{o}mez \cite{DLGA} even for the elliptic counterpart of \eqref{1.1}, and of some previous findings of the authors in \cite{AALG} and \cite{AALG2}. Then, it sharpens some results  of \cite{LG20} by enlarging the class of critical degenerate weight functions for which \eqref{1.1} admits a positive periodic solution in  an unbounded interval of values of the parameter $\l$.  The latest findings, not having any previous elliptic counterpart, are utterly new and of great interest in Population Dynamics.

\vspace{0.1cm}

\noindent \emph{2010 MSC:}   35K55, 35K57, 35B09, 35B10, 35Q92.
\vspace{0.1cm}

\noindent \emph{Keywords and phrases: Periodic-Parabolic logistic equation, positive periodic solutions, degenerate problem, weighted problem. }
\vspace{0.1cm}

\noindent This paper has been supported by the IMI of Complutense University and the Ministry of
Science and Innovation of Spain under Grant PGC2018-097104-B-IOO.
\end{abstract}

\vspace{0.2cm}

\section{Introduction}

\noindent The main goal of this paper is to characterize the existence of classical positive periodic solutions, $(\l,u)$, of the periodic-parabolic problem
\begin{equation}
\label{1.1}
  \left\{ \begin{array}{ll} \p_t u +\mf{L} u  = \l m(x,t) u-a(x,t) f(u)u & \quad \hbox{in}\;\; Q_T:=\O\times (0,T), \\[1ex] \mc{B}u =0  & \quad \hbox{on}\;\; \p\O\times [0,T],\\  \end{array} \right.
\end{equation}
where:
\begin{itemize}
\item[($H_\Omega$)] $\O$ is a bounded subdomain (open and connected subset) of $\R^N$, $N\geq 1$, of class $\mathcal{C}^{2+\t}$ for some $0<\t\leq 1$.
\item[($H_\mf{L}$)] $\mf{L}$ is a non-autonomous differential operator of the form
\begin{equation}
\label{1.2}
  \mathfrak{L}\equiv \mathfrak{L}(x,t):= -\sum_{i,j=1}^N a_{ij}(x,t)\frac{\p^2}
  {\p x_i \p x_j}+\sum_{j=1}^N b_j(x,t) \frac{\p}{\p x_j}+c(x,t)
\end{equation}
with $a_{ij}=a_{ji}, b_j, c \in F$ for all $1\leq i, j\leq N$ and some $T>0$, where
\begin{equation}
\label{1.3}
  F:= \left\{u\in \mathcal{C}^{\t,\frac{\t}{2}}(\bar\O\times \R;\R)\;: \;
  u(\cdot,T+t)=u(\cdot,t) \;\; \hbox{for all}\; t \in\R \right\}.
\end{equation}
Moreover, $\mathfrak{L}$ is assumed to be uniformly elliptic in
$\bar Q_T=\bar \O \times [0,T]$, i.e., there exists a constant $\mu>0$ such that
\begin{equation*}
  \sum_{i,j=1}^N a_{ij}(x,t) \xi_i\xi_j\geq \mu\, |\xi|^2\quad \hbox{for all} \;\; (x,t,\xi)\in \bar Q_T\times \R^N,
\end{equation*}
where $|\cdot|$ stands for the  Euclidean norm of $\R^N$.
\item[($H_\mc{B}$)] $\mc{B}$ is a non-classical mixed boundary operator of the form
\begin{equation}
\label{1.4}
    \mc{B} \xi := \left\{ \begin{array}{ll}
    \xi \qquad & \hbox{on } \;\;\G_0, \\     \p_{\nu} \xi + \b(x) \xi \qquad &
   \hbox{on } \;\; \G_1,  \end{array} \right.\qquad \xi \in \mathcal{C}(\G_0)\oplus\mathcal{C}^1(\O\cup \G_1),
\end{equation}
where $\G_0$ and $\G_1$ are two disjoint open and closed subsets of $\partial \O$ such that $\partial \O:= \G_0\cup \G_1$. In \eqref{1.4}, $\b\in \mathcal{C}^{1+\t}(\G_1)$ and $\nu =(\nu_1,...,\nu_N)\in \mathcal{C}^{1+\t}(\p\O;\R^N)$ is an outward pointing nowhere tangent vector field.
\item[($H_{a}$)] $m,a\in F$ and $a$ satisfies $a\gneq0$, i.e., $a\geq 0$ and $a\neq0$.
\item[($H_f$)] $f\in \mc{C}^1(\R;\R)$ satisfies $f(0)=0$, $f'(u)>0$ for all $u>0$, and
\begin{equation*}
  \lim_{u\to +\infty}f(u)=+\infty.
\end{equation*}
\end{itemize}
Thus, we are working under the general setting of \cite{ALGR,ALGN}, where  $\b$ can change sign, which remains outside the  classical framework  of Hess \cite{Hess} and Daners and L\'{o}pez-G\'{o}mez \cite{DLG}, where $\b \geq 0$.  In this paper, $\mc{B}$ is the \emph{Dirichlet boundary operator} on $\G_0$, and the \emph{Neumann}, or a \emph{first order regular oblique derivative boundary operator}, on $\G_1$, and either $\G_0$, or $\G_1$, can be empty. When $\G_1=\emptyset$, $\mc{B}$ becomes the Dirichlet boundary operator, denoted by $\mc{D}$. Thus, our results are  sufficiently general as to cover the existence results of Du and Peng \cite{DuPeng12,DuPeng13}, where $m(x,t)$ was assumed to be constant, $\G_0= \emptyset $ and $\b =0$, and of Peng and Zhao \cite{PZ}, where $\b\geq 0$.  The importance of combining temporal periodic heterogeneities in spatially heterogeneous models can be easily realized by simply having a look at the Preface of \cite{LG15}.  
\par
By a classical solution of \eqref{1.1} we mean a solution pair $(\l,u)$, or simply $u$, with $u\in E$, where
\begin{equation*}
  E:= \left\{u\in \mathcal{C}^{2+\t,1+\frac{\t}{2}}(\bar\O\times \R;\R)\;: \;  u(\cdot,T+t)=u(\cdot,t) \;\; \hbox{for all}\;\; t \in\R \right\}.
\end{equation*}
The first  goal of this paper is to obtain a periodic-parabolic counterpart of \cite[Th. 1.1]{DLGA} to characterize the existence of classical positive periodic-solutions of \eqref{1.1}, regardless
the structure of $a^{-1}$.  Our main result generalizes, very substantially, our previous findings in \cite{DLGA}, where we dealt with the special case  $m\equiv 1$, and \cite{AALG2}, where we imposed some (severe) restrictions on the structure of $a^{-1}(0)$. All previous requirements are removed in this paper. To state our result, we need to introduce some of notation.
\par
Throughout this paper, $\mc{P}$ stands for the periodic-parabolic operator
\begin{equation*}
  \mathcal{P}:= \p_t + \mathfrak{L},
\end{equation*}
and, for every tern $(\mathcal{P},\mc{B},Q_T)$ satisfying the previous general assumptions, we denote by $\sigma[\mathcal{P},\mc{B},Q_T]$ the principal eigenvalue of $(\mathcal{P},\mc{B},Q_T)$, i.e., the unique value of $\tau$ for which the linear  periodic-parabolic eigenvalue problem
\begin{equation*}
  \left\{ \begin{array}{ll}  \mathcal{P}\varphi  = \tau \varphi & \quad \hbox{in}\;\; Q_T, \\[1ex] \mc{B}\varphi =0  & \quad \hbox{on}\;\; \p\O \times [0,T],\\  \end{array} \right.
\end{equation*}
admits a positive eigenfunction $\varphi\in E$. In such case, $\v$ is unique, up to a positive
 multiplicative constant, and $\varphi\gg 0$ in the sense that, for every $t\in [0,T]$,
\begin{equation*}
  \v(x,t)>0 \;\; \hbox{for all}\;\; x \in \O\cup\G_1\;\;\hbox{and}\;\;
  \p_{\nu} \v(x,t)<0 \;\;\hbox{if}\;\; x \in \G_0.
\end{equation*}
The existence of the principal eigenvalue in the special case when $\beta\geq 0$ is a classical result attributable to Beltramo and Hess \cite{BH}. In the general case when $\beta$ changes of sign this result has been  established, very recently, in \cite{ALGR,ALGN}. Throughout this paper, we will also consider  the principal eigenvalue
\begin{equation}
\label{1.5}
\Sigma(\lambda,\gamma):=\sigma[\mathcal{P}-\lambda m+\gamma a, \mc{B}, Q_T],\qquad \l,\,\gamma\in\R.
\end{equation}
Since $a\gneq 0$, by Proposition \ref{pr21}(a), $\Sigma(\l,\g)$ is an increasing function of $\g$ and hence, for every $\l\in\R$, the limit
\begin{equation}
\label{1.6}
\Sigma(\l,\infty):=\lim_{\g\ua \infty}\Sigma(\l,\g)\leq \infty
\end{equation}
is well defined. The first aim of this paper is to establish the following characterization.

\begin{theorem}
\label{th11}
Suppose ($H_\Omega$), ($H_\mf{L}$), ($H_\mc{B}$), ($H_{a}$) and ($H_f$). Then, \eqref{1.1} admits a
positive solution if, and only if,
\begin{equation}
\label{1.7}
  \Sigma(\l,0)<0<\Sigma(\l,\infty).
\end{equation}
Moreover, it is unique, if it exists.
\end{theorem}

As Theorem \ref{th11} holds regardless the structure of the compact set $a^{-1}(0)$,
it provides us with a substantial extension of \cite[Th. 1.1]{DLGA}, which was established for the elliptic counterpart of \eqref{1.1},  and some of the findings of the authors in \cite{AALG} and \cite{AALG2}, where some additional (severe) restrictions on the structure of $a^{-1}(0)$  were imposed.
\par
From the point of view of the applications, determining whether or not $\Sigma(\l,\infty)=\infty$ occurs, is imperative to ascertain the hidden structure of the set of $\l\in\R$ for which
\eqref{1.7} holds. Thanks to Theorem \ref{th11}, this is the set of $\l$'s for which
\eqref{1.1} possesses a positive periodic solution. To illustrate what we mean, suppose that $a(x,t)$ is positive and separated away from zero in $Q_T$, i.e.,
$$
  a(x,t)\geq \o>0 \quad \hbox{for all}\;\; (x,t)\in Q_T.
$$
Then, by Proposition \ref{pr21}(a), it is apparent that
$$
  \Sigma(\lambda,\gamma):=\sigma[\mathcal{P}-\lambda m+\gamma a, \mc{B}, Q_T]\geq
  \sigma[\mathcal{P}-\lambda m, \mc{B}, Q_T]+\g \o
$$
and hence, for every $\l\in\R$,
$$
  \Sigma(\l,\infty)=\lim_{\g\ua \infty}\Sigma(\l,\g)=\infty.
$$
Therefore, in this case, \eqref{1.1} has a positive solution if, and only if,
$\Sigma(\l,0)<0$. Naturally, to determine the structure of the set of $\l$'s where
$$
  \Sigma(\l,0)=\sigma[\mathcal{P}-\lambda m, \mc{B}, Q_T]<0
$$
the nodal properties of $m(x,t)$ play a pivotal role. Indeed, whenever $m\neq 0$ has constant sign, $\Sigma(\l,0)<0$ on some $\l$-interval which might become $\R$. However, when
\begin{equation}
\label{1.8}
\int_0^T \min_{x\in \bar \O}m(x,t)\,dt < 0 < \int_0^T \max_{x\in \bar\O}m(x,t)\,dt
\end{equation}
and $\Sigma(\tilde \l,0)>0$ for some $\tilde \l\in\R$, then
$$
   \{\l\in\R\;:\;\Sigma(\l,0)<0\}= (-\infty,\l_-)\cup(\l_+,\infty),
$$
where $\l_\pm$ stand for the
principal eigenvalues of the linear weighted eigenvalue problem
\begin{equation*}
  \left\{ \begin{array}{ll}  \mathcal{P}\varphi  = \lambda m(x,t) \varphi & \quad \hbox{in}\;\; Q_T, \\[1ex] \mc{B}\varphi =0  & \quad \hbox{on}\;\; \p\O \times [0,T].\\  \end{array} \right.
\end{equation*}
A complete analysis of this particular issue was carried out in \cite{ALGR} and \cite{ALGN}. So, we stope this discussion here.
\par
Contrarily to what happens in the autonomous case, $\Sigma(\l,\infty)$ can be infinity if
$a^{-1}(0)$  has non-empty interior in $Q_T$. The next result characterizes whether $\Sigma(\l,\infty)<\infty$ holds for a special family of
weight functions $a(x,t)$ with a tubular vanishing set, $a^{-1}(0)$, as described on \cite[Th. 6.4]{LG20}. This result generalized, very substantially, some pioneering findings of Daners and Thornet \cite{DT}.

\begin{theorem}
\label{th12}
$\Sigma(\l,\infty)<\infty$ if, and only if, there is a continuous map $\tau :[0,T]\to \O$ such that $\tau(0)=\tau(T)$ and
\begin{equation*}
  (\tau(t),t)\in \mathrm{int\,} a^{-1}(0)\quad \hbox{for all}\;\; t \in [0,T],
\end{equation*}
i.e., if we can advance, upwards in time, from time $t=0$ up to time $t=T$  within the interior of the vanishing set of $a(x,t)$.
\end{theorem}

Actually, the fact that $\Sigma(\l,\infty)<\infty$ if $\tau$ exists is a general
feature not depending on the particular structure of $a^{-1}(0)$. In the proof of \cite[Th. 4.1]{LG20}
it was derived from a technical device introduced in the proof of Lemma 15.4 of Hess \cite{Hess}. The converse result that $\Sigma(\l,\infty)=\infty$ when the map $\tau$ does not exist, has been already proven  when $a^{-1}(0)$ has the structure sketched in Figure \ref{Fig1}, thought it remained an open problem
in \cite{LG20} for more general situations. In Figure \ref{Fig1}, as in all remaining figures of this paper, the $x$-variables have been plotted in abscissas, and time $t$ in ordinates. The white region
represents $a^{-1}(0)$, while the dark one is the set of $(x,t)\in Q_T$ such that $a(x,t)>0$.

\begin{figure}[h!]
\centering
\includegraphics[scale=0.8]{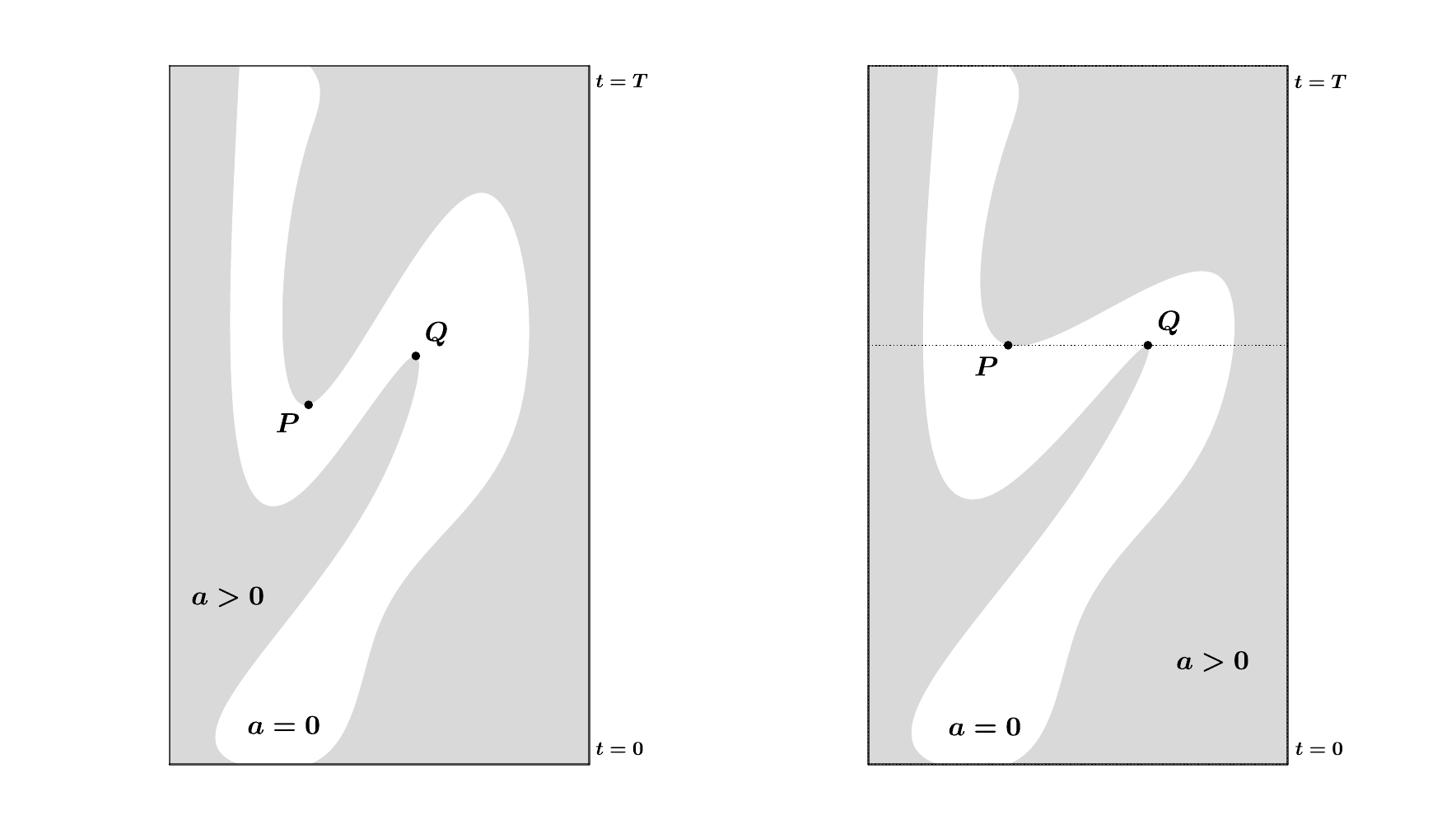}
\vspace{-0.5cm}
\caption{Two simple situations where $\Sigma(\l,\infty)=\infty$; the second one is critical. }
\label{Fig1}
\end{figure}

In  Figure \ref{Fig1}, we are denoting $P=(x_P,t_P)$ and $Q=(x_Q,t_Q)$. The first picture shows
a genuine situation where $t_P<t_Q$, while $t_P=t_Q$ in the second one, which is a limiting case with respect to the complementary case when $t_P>t_Q$, where, due to Theorem \ref{th12}, $\Sigma(\l,\infty)<\infty$. All pictures throughout this paper idealize the general case when $N\geq 1$, though represent some special cases with $N=1$.
\par
Subsequently, as sketched by Figure \ref{Fig1}, we are assuming
\begin{equation}
\label{1.9}
  a(x,t)>0\quad \hbox{for all}\;\; (x,t)\in \p\O \times [0,T].
\end{equation}
Our second goal in this paper is to extend Theorem \ref{th12} to cover a more general class of critical weight functions $a(x,t)$ satisfying \eqref{1.9} for which the map $\tau$ does not exist. These cases are critical in the sense that some small perturbation of $a(x,t)$ might, or might not, entail the existence of $\tau$. To precise what we mean, let $a_0\in F$ be an admissible weight function such that
$\mathrm{int\,}a_0^{-1}(0)$ has the structure of the first plot of Figure \ref{Fig2} with $t_Q<t_P$. In such case, by Theorem \ref{th12}, $\Sigma(\l,\infty)<\infty$ for all $\l\in\R$.

\begin{figure}[h!]
\centering\includegraphics[scale=0.8]{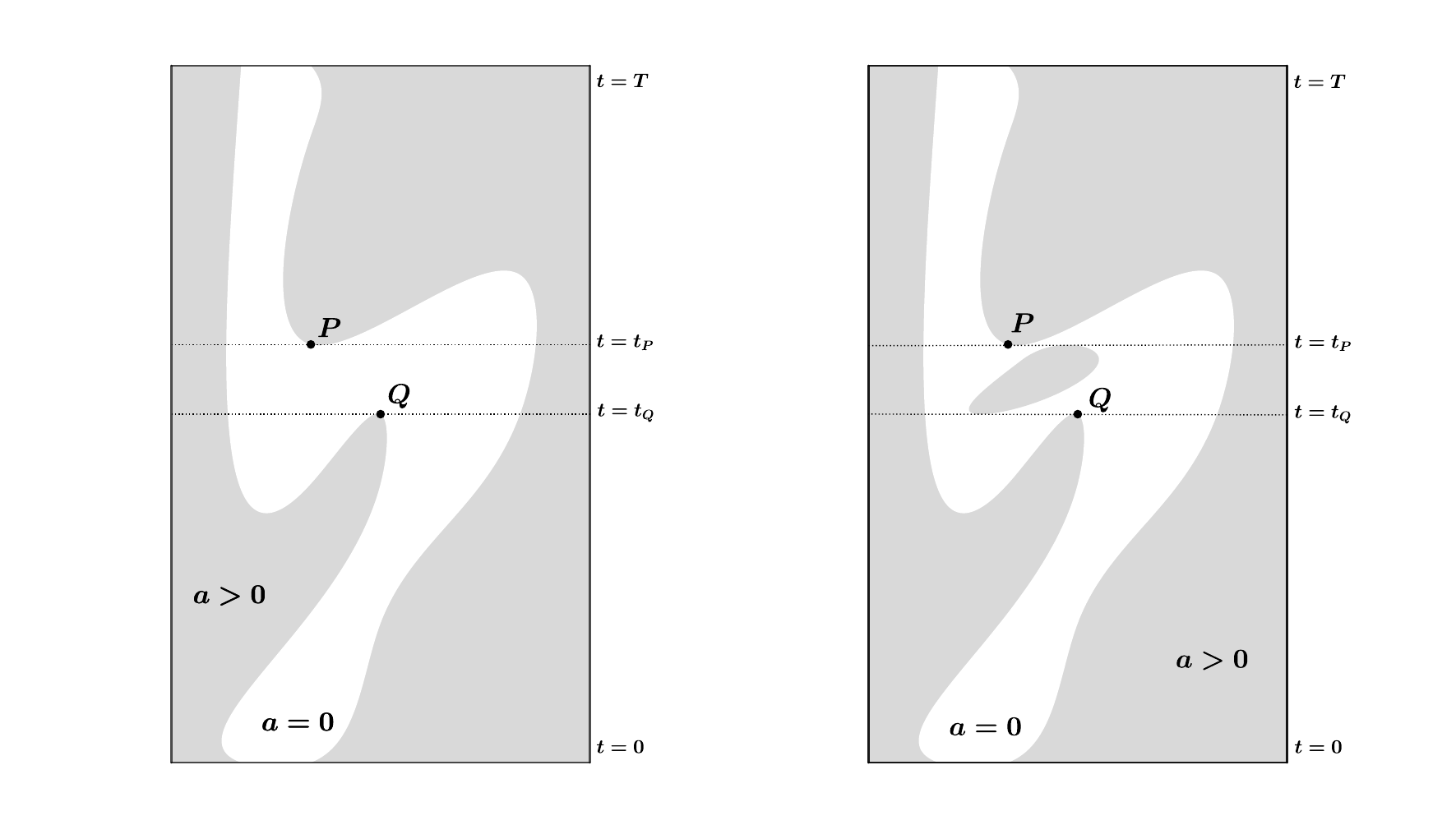}
\caption{The construction of the weight function $a=a_0+a_1$.}
\label{Fig2}
\end{figure}

Now, let $\mc{O}$ be an open and connected subset of $\mathrm{int\,}a_0^{-1}(0)$, with sufficiently smooth boundary, such that
$$
  M:=\bar{\mc{O}}\subset \mathrm{int\,} a_0^{-1}(0),
$$
and
$$
  t_Q=\min \{t\in [0,T]\;:\;M_t\neq \emptyset\}<
  t_P= \max\{t\in [0,T]\;:\;M_t\neq \emptyset\},
$$
where, for any subset $A\subset \bar\O \times [0,T]$ and $t\in [0,T]$, $A_t$ stands for the $t$-slice
$$
  A_t:=\{x\in\bar \O\;:\;\; (x,t)\in A\}.
$$
The  right picture of Figure \ref{Fig2} shows an admissible $M$ placed between the lines
$t=t_P$ and $t=t_Q$. Lastly, we consider any function $a_1\in F$ such that
$$
  a_1(x,t)>0\quad \hbox{if and only if}\;\; (x,t)\in\mathcal{O}.
$$
Then, the next result holds.

\begin{theorem}
\label{th13}
Suppose that $a=a_0+a_1$ and there is not a continuous map $\tau :[0,T]\to \O$ such that $\tau(0)=\tau(T)$ and $(\tau(t),t)\in \mathrm{int\,} a^{-1}(0)$ for all $t \in [0,T]$. Then, $\Sigma(\l,\infty)=\infty$ for all $\l\in\R$.
\end{theorem}

Note that the relative position of the component $M$ with respect to the remaining components where $a>0$ is crucial to determine whether or not $\Sigma(\l,\infty)$ is finite. Figure \ref{Fig3} shows a magnification of the most significative piece of the right picture of Figure \ref{Fig2} in a case where the curve $\tau$ does not exist (on the left), together with another situation where the curve $\tau$ does exist (on the right). In the second case, according to Theorem \ref{th12}, we have that  $\Sigma(\l,\infty)<\infty$ for all $\l\in\R$.

\begin{figure}[h!]
\centering\includegraphics[scale=0.8]{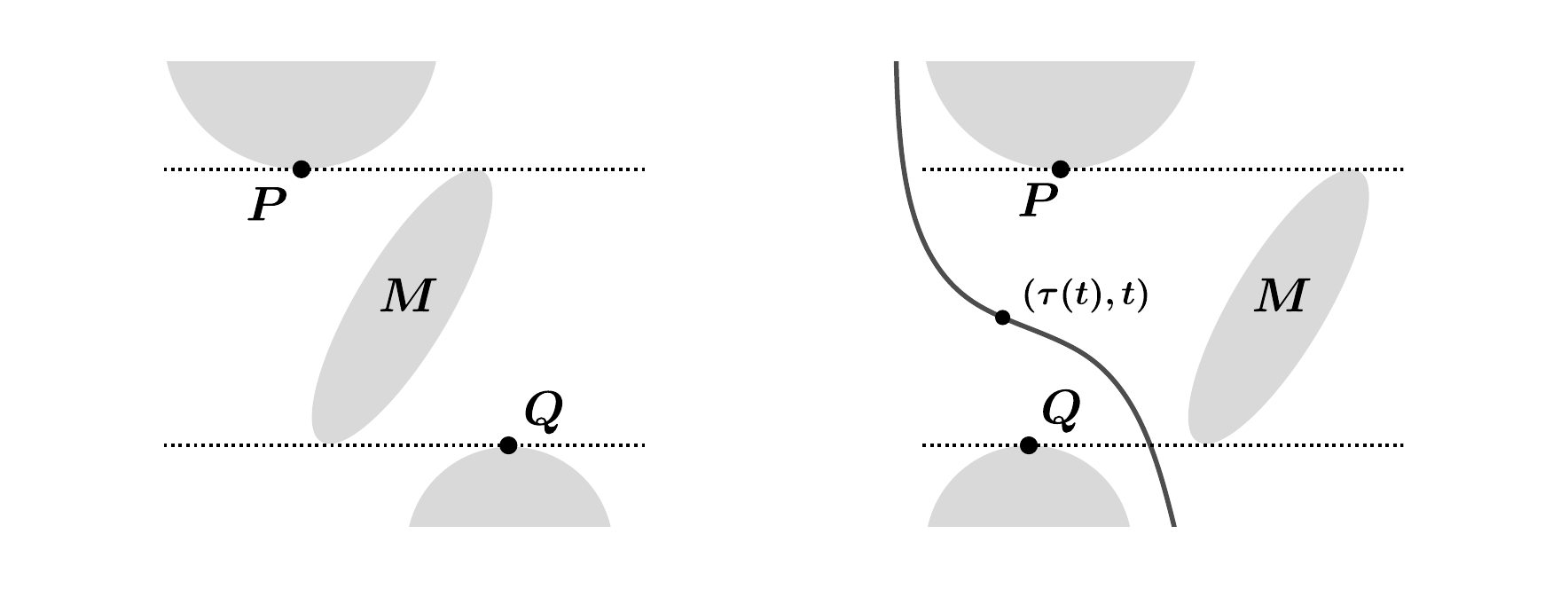}\\
 \caption{Influence of the relative position of $M$.}
\label{Fig3}
\end{figure}

Theorem \ref{th13} can be generalized up to cover a wider class of critical weight functions $a(x,t)$. Indeed, suppose that the interior of $a_0^{-1}(0)$ looks like shows
the first picture of Figure \ref{Fig2}, with $t_Q<t_P$, and, for any given integer $n\geq 2$, let
$t_i \in (t_Q,t_P)$, $i\in\{1,...,t_{n-1}\}$, such that
\begin{equation*}
  t_Q<t_1<t_2<\cdots<t_{n-1}<t_P.
\end{equation*}
Now, consider $n$ open subsets, $\mc{O}_j$, $1\leq j\leq n$, with sufficiently smooth boundaries and mutually disjoint closures, such that
\begin{equation*}
   M_j:=\bar{\mc{O}_j}\subset \mathrm{int\,}a_0^{-1}(0),\qquad 1\leq j\leq n,
\end{equation*}
and
\begin{equation*}
\begin{split}
t_Q & = \min\{t\in [0,T]\;:\; (M_1)_t\neq \emptyset\} \\ & < t_1=
\max\{t\in [0,T]\;:\; (M_1)_t\neq \emptyset\} = \min \{t\in [0,T]\;:\; (M_2)_t\neq \emptyset\}\\ &
< t_2 = \max\{t\in [0,T]\;:\; (M_2)_t\neq \emptyset\} = \min \{t\in [0,T]\;:\; (M_3)_t\neq \emptyset\}\\ & < t_3  = \max\{t\in [0,T]\;:\; (M_3)_t\neq \emptyset\} = \min \{t\in [0,T]\;:\; (M_4)_t\neq \emptyset\}\\ & < \cdots < t_{n-1}=\max\{t\in [0,T]\;:\; (M_{n-1})_t\neq \emptyset\}\\ &  \hspace{2.3cm}
= \min \{t\in [0,T]\;:\; (M_n)_t\neq \emptyset\}=t_P.
\end{split}
\end{equation*}
Lastly, let $a_i\in F$ be, $i\in\{1,...,n\}$, such that $a_i(x,t)>0$ if and only if $(x,t)\in \mc{O}_i$. Then, the next generalized version of Theorem \ref{th13} holds.

\begin{theorem}
\label{th14}
Suppose that
\begin{equation*}
 a= a_0+a_1+\cdots+a_{n}
\end{equation*}
and that there is not a continuous map $\tau :[0,T]\to \O$ such that $\tau(0)=\tau(T)$ and
$(\tau(t),t)\in \mathrm{int\,} a^{-1}(0)$ for all $t \in [0,T]$. Then, $\Sigma(\l,\infty)=\infty$ for all $\l\in\R$.
\end{theorem}

The left picture of Figure \ref{Fig4} illustrates a typical situation with $n=3$ where all the
assumptions of Theorem \ref{th14} are fulfilled. In the case illustrated by the right picture one can construct an admissible curve $\tau$ and hence, due to Theorem \ref{th12},
$\Sigma(\l,\infty)<\infty$ for all $\l\in\R$. This example shows that the non-existence of $\tau$ is necessary for the validity of Theorem \ref{th14}.

\begin{figure}[h!]
\centering\includegraphics[scale=0.8]{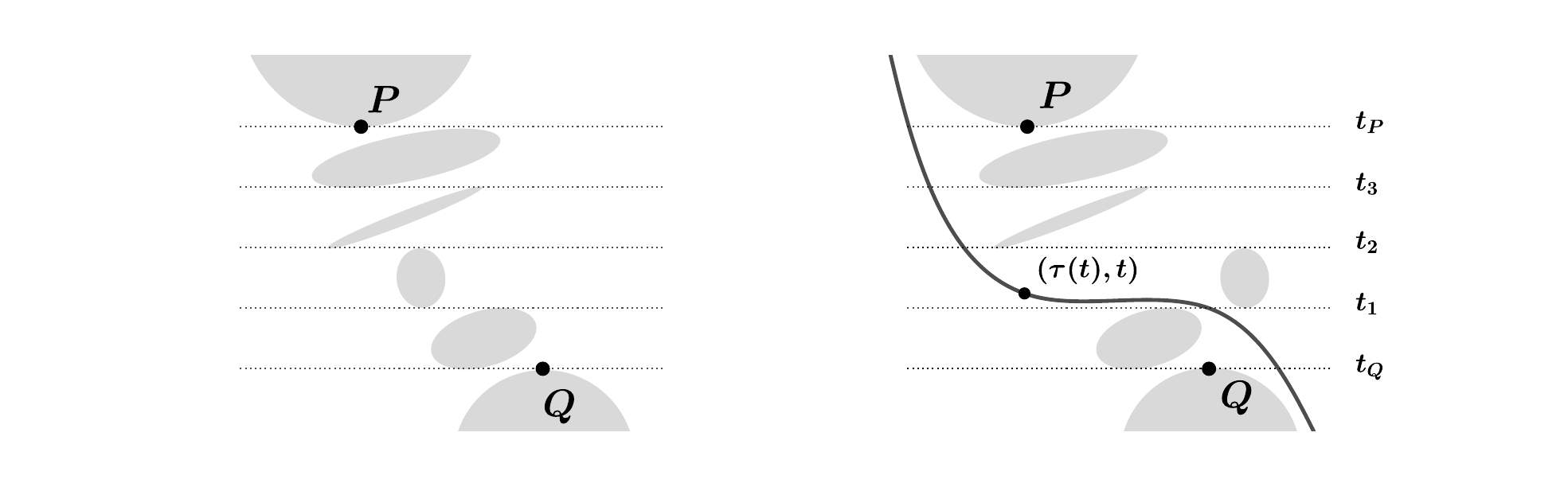}\\
 \caption{Two admissible configurations with $n=4$.}
\label{Fig4}
\end{figure}

This paper is organized as follows: Section 2 contains some preliminaries, Section 3 analyzes some general properties of the bi-parametric curve of principal eigenvalues $\Sigma(\l,\gamma)$, Section 4 analyzes the continuous dependence of $\s[\mc{P},\mc{B},Q_T]$ with respect to the variations of
$\O$ along $\G_0$, Section 5 delivers the proof of Theorem \ref{th11} and, finally, Section 6 delivers the proof of Theorems \ref{th13} and \ref{th14}.

\section{Preliminaries}

\noindent The principal eigenvalue $\s[\mc{P},\mc{B},Q_T]$ satisfies the next pivotal characterization.
It is \cite[Th. 1.2]{ALGN}.

\begin{theorem}
\label{th21}
The following conditions are equivalent:
\begin{itemize}
\item[(a)] $\s[\mc{P},\mc{B},Q_T]>0$.

\item[(b)] $(\mathcal{P},\mc{B},Q_T)$ possesses  a positive strict supersolution $h\in E$.

\item[(c)]  Any strict supersolution $u\in E$ of
$(\mathcal{P},\mc{B},Q_T)$ satisfies $u\gg 0$ in the sense that, for every $t\in [0,T]$, $u(x,t)>0$ for all $x\in\O\cup\G_1$, and
\begin{equation*}
  \p_\nu u (x,t)<0\;\; \hbox{for all }\;\; x \in u^{-1}(0)\cap \G_0.
\end{equation*}
In other words, $(\mc{P},\mc{B},Q_T)$ satisfies the strong maximum principle.
\end{itemize}
\end{theorem}

Theorem \ref{th21} is the  periodic-parabolic counterpart of Theorem 2.4 of
Amann and L\'{o}pez-G\'{o}mez \cite{AL98}, which goes back to L\'{o}pez-G\'{o}mez and
Molina-Meyer \cite{LGMM} for cooperative systems under Dirichlet boundary conditions, and to
\cite{ALG94} for cooperative periodic-parabolic systems under  Dirichlet boundary conditions. In  \cite{ALGR} the most fundamental properties of the principal eigenvalue $\s[\mc{P},\mc{B},Q_T]$ were
derived from Theorem \ref{th21}. Among them, its uniqueness, algebraic simplicity  and strict dominance, as well as the properties collected in the next result, which is a direct consequence of Propositions 2.1 and 2.6 of \cite{ALGR}.

\begin{proposition}
\label{pr21} The principal eigenvalue satisfies the following properties:
\begin{enumerate}
\item[{\rm (a)}] Let $V_1, V_2\in F$ be such that $V_1 \lneq V_2$. Then,
\[
  \s[\mc{P}+V_1,\mc{B},Q_T]<\s[\mc{P}+V_2,\mc{B},Q_T].
\]
\item[{\rm (b)}] Let $\O_0$ be a subdomain of $\O$ of class $\mc{C}^{2+\t}$ such that $\bar\O_0\subset \O$. Then,
\[
   \s[\mc{P},\mc{B},Q_T]< \s[\mc{P},\mc{D},\O_0\times (0,T)].
\]
\end{enumerate}
\end{proposition}

The principal eigenvalue
\begin{equation*}
\Sigma(\lambda):=\sigma[\mathcal{P}-\lambda m,\mc{B},Q_T],\qquad \l\in\R,
\end{equation*}
plays a pivotal role in characterizing the existence of positive solutions of \eqref{1.1}. According to Theorems 5.1 and 6.1 of \cite{ALGR}, $\Sigma(\l)$ is analytic, strictly concave and, as soon as $m(x,t)$ satisfies \eqref{1.8},
\begin{equation*}
\lim_{\lambda\to \pm \infty} \Sigma(\lambda)=-\infty.
\end{equation*}
Therefore, if \eqref{1.8} holds and $\Sigma(\l_0)>0$ for some $\l_0\in\R$, then there exist $\lambda_-<\lambda_0<\lambda_+$ such that
\begin{equation}
\label{2.1}
 \Sigma(\lambda)\left\{\begin{array}{ll} <0 &\quad \hbox{if} \;\; \lambda \in (-\infty, \lambda_-)\cup(\lambda_+,\infty), \cr =0 &\quad  \hbox{if} \;\; \lambda\in\{\lambda_-,\l_+\},\cr
 >0 &\quad \hbox{if}\;\; \lambda \in ( \lambda_-,\lambda_+). \end{array} \right.
\end{equation}
Actually, one can choose $\lambda_0$ such that $\Sigma'(\lambda_0)=0$, $\Sigma'(\lambda)>0$ if $\lambda<\lambda_0$ and $\Sigma'(\lambda)<0$ if $\lambda>\lambda_0$.  Figure \ref{Fig5} shows the graph of $\Sigma(\l)$ for this choice in such case.

\begin{figure}[h!]
\centering
\includegraphics[scale=1]{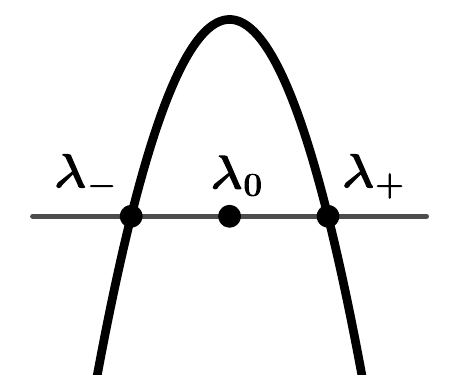}
\caption{The graph of $\Sigma(\lambda)$. }
\label{Fig5}
\end{figure}
\par
The next result follows from Lemma 3.1 and Theorem 4.1 of \cite{AALG2}. It provides us with a necessary condition for the existence of positive periodic solutions of \eqref{1.1}, as well as with its uniqueness.

\begin{theorem}
\label{th22}
Suppose that $a\gneq 0$ in $Q_T$ and that \eqref{1.1} admits a positive periodic solution, $(\l,u)$, with $u\in E$. Then, $\Sigma(\lambda)<0$. Moreover, $u$ is the unique positive periodic solution of \eqref{1.1} and $u\gg 0$, in the sense that
 \begin{equation*}
  u(x,t)>0 \;\; \hbox{for all}\;\; x\in \O\cup\G_1\;\;\hbox{and}\;\;
  \p_{\nu} u(x,t)<0 \;\;\hbox{for all}\;\; x\in \G_0.
\end{equation*}
 \end{theorem}

The next result establishes that $\Sigma(\l)<0$ is not only necessary but also sufficient whenever
$a^{-1}(0)=\emptyset$. It is Theorem 6.1 of \cite{AALG2}.

\begin{theorem}
\label{th23}
Suppose that $a(x,t)>0$ for all $(x,t)\in \bar{Q}_T$. Then, \eqref{1.1} possesses a positive periodic solution if, and only if, $\Sigma(\lambda)<0$.
\end{theorem}

\section{A bi-parametric family of principal eigenvalues}

\noindent The next result collects some important properties of the bi-parametric family of principal eigenvalues $\Sigma(\l,\g)$ introduced in \eqref{1.5}.

\begin{theorem}
\label{th31}
The function $\Sigma(\lambda,\gamma)$ is concave, analytic in $\l$ and $\g$, and increasing with respect to $\g\in\R$. Thus, for every $\l\in\R$, the limit $\Sigma(\l,\infty)$ introduced in \eqref{1.6}
is well defined and, for every $\g>0$, satisfies
\begin{equation}
\label{iii.1}
\Sigma(\lambda)=\Sigma(\l,0)<\Sigma(\l,\g)<\Sigma(\lambda,\infty) \quad \hbox{for all} \;\; \lambda \in \R.
\end{equation}
Moreover, one of the following excluding options occurs. Either
\begin{enumerate}
\item[{\rm (a)}] $\Sigma(\lambda,\infty)=\infty$ for all $\lambda\in \R$; or
\item[{\rm (b)}] $\Sigma(\lambda,\infty)<\infty$ for all $\lambda \in \R$.
\end{enumerate}
Furthermore, when {\rm (b)} occurs, also the function $\Sigma(\cdot, \infty)$ is concave.
\end{theorem}
\begin{proof}
The fact that $\Sigma(\lambda,\gamma)$ is analytic and concave with respect to each of the parameters $\l$ and $\gamma$ follows from \cite[Th. 5.1]{ALGR}. By Proposition \ref{pr21}(a),
it is increasing with respect of $\g\in\R$. Thus, for every $\l\in\R$, $\Sigma(\l,\infty)\in\R\cup\{+\infty\}$ is well defined and \eqref{iii.1} holds.
\par
Now, suppose that $\Sigma(\hat{\lambda},\infty)<\infty$ for some $\hat{\lambda}\in \R$. Then, according to Proposition \ref{pr21}(a),  we have that,  for every $\l, \g \in\R$,
\begin{align*}
\Sigma(\lambda, \gamma)&=\s[\mc{P}-\l m +\g a,\mc{B},Q_T]\\ & =
\sigma[\mathcal{P}-\hat{\lambda} m+\gamma a+(\hat{\lambda}-\lambda)m, \mc{B}, Q_T]\\ & \leq \Sigma(\hat{\lambda},\gamma)+|\lambda-\hat{\lambda}|\|m\|_{\infty}<+\infty.
\end{align*}
Consequently, the option (b) occurs.
\par
As, for every $\g\in\R$, the function $\Sigma(\l,\g)$ is concave, it is apparent that, for every
$\l_1, \l_2, \g  \in \R$ and $t \in [0,1]$,
\begin{equation*}
  \Sigma(t \l_1+(1-t)\l_2,\g)\geq t \,\Sigma(\l_1,\g)+(1-t) \,\Sigma(\l_2,\g).
\end{equation*}
Thus, by \eqref{iii.1}, for every $\g\in\R$,
\begin{equation*}
  \Sigma(t \l_1+(1-t)\l_2,\infty) > t \,\Sigma(\l_1,\g)+(1-t)\, \Sigma(\l_2,\g).
\end{equation*}
Consequently, letting $\g \ua\infty$, it follows that
\begin{equation*}
  \Sigma(t \l_1+(1-t)\l_2,\infty) \geq  t \,\Sigma(\l_1,\infty)+(1-t)\, \Sigma(\l_2,\infty),
\end{equation*}
which ends the proof.
\end{proof}

\section{Continuous dependence of $\s[\mc{P},\mc{B},Q_T]$ with respect to $\G_0$}

\noindent In this section, we establish the continuous dependence of $\sigma[\mathcal{P},\mc{B},Q_T]$ with respect to a canonical perturbation of the domain $\O$ around its Dirichlet boundary $\Gamma_0$. Our result provides with a periodic counterpart of \cite[Th. 8.5]{LG13} for the perturbation
\begin{equation}
\label{iv.1}
\mathcal{O}_n:=\O \cup \Big(\Gamma_0+B_{\frac{1}{n}}(0)\Big)=\O \cup \left\{x \in \R^N\;\; : \;\; \hbox{dist}(x,\Gamma_0)<\tfrac{1}{n}\right\}.
\end{equation}
For sufficiently large $n\geq 1$, say $n\geq n_0$, we consider $a_{ij,n}=a_{ji,n}, b_{j,n}, c_n \in F_n$, $1\leq i, j\leq N$, such that $a_{ij,n}=a_{ij}$, $b_{j,n}=b_j$ and $c_n=c$ in $Q_T$, where we are denoting
\begin{equation}
\label{iv.2}
  F_n:= \left\{u\in \mathcal{C}^{\t,\frac{\t}{2}}(\overline{\mc{O}}_n\times \R;\R)\;: \;\;
  u(\cdot,T+t)=u(\cdot,t) \;\; \hbox{for all}\; t \in\R \right\},
\end{equation}
as well as the associated differential operators
\begin{equation}
\label{iv.3}
  \mathfrak{L}_n\equiv \mathfrak{L}_n(x,t):= -\sum_{i,j=1}^N a_{ij,n}(x,t)\frac{\p^2}
  {\p x_i \p x_j}+\sum_{j=1}^N b_{j,n}(x,t) \frac{\p}{\p x_j}+c_n(x,t),
\end{equation}
the  periodic parabolic operators $\mc{P}_n:=\p_t+\mf{L}_n$, $n\geq n_0$, and the boundary operators
$\mc{B}_n$ defined by
\begin{equation}
\label{iv.4}
\mc{B}_n:= \left\{ \begin{array}{ll} \mc{D} \;\; & \hbox{on} \;\; \Gamma_{0,n}\equiv \partial \mathcal{O}_n \setminus \Gamma_1, \cr \mc{B} \;\; & \hbox{on} \;\; \G_{1,n} \equiv \G_1.\end{array}\right.
\end{equation}
Then, the next result holds.

\begin{theorem}
\label{th41}
Let $\mathcal{O}_n$ be defined by \eqref{iv.1} for all $n\geq n_0$, and set $Q_{T,n}:=\mc{O}_n\times (0,T)$. Then,
\begin{equation}
\label{iv.5}
\lim_{n\to\infty} \sigma[\mathcal{P}_n,\mc{B}_n,Q_{T,n}]=\sigma[\mathcal{P},\mc{B},Q_T].
\end{equation}
\end{theorem}
\begin{proof}
By Proposition 2.6 of \cite{ALGR}, we have that, for every $n\geq n_0$,
\begin{equation*}
\sigma_n:=\sigma[\mathcal{P},\mathcal{B}_n,Q_{T,n}]< \sigma[\mathcal{P},\mathcal{B}_{n+1},Q_{T,n+1}]<\sigma[\mathcal{P},\mathcal{B},Q_T].
\end{equation*}
Thus, the limit in \eqref{iv.5} exists and it satisfies that
\begin{equation}
\label{iv.6}
\sigma_\infty:=\lim_{n\to\infty} \sigma_n\leq \sigma[\mathcal{P},\mathcal{B},Q_T].
\end{equation}
Thus, it suffices to show that
\begin{equation}
\label{iv.7}
  \sigma_\infty=  \sigma[\mathcal{P},\mathcal{B},Q_T].
\end{equation}
To prove \eqref{iv.7} we can argue as follows. For every $n\in \N$, let $\varphi_n\gg 0$ denote the (unique) principal eigenfunction associated with $\sigma_n$ normalized so that
\begin{equation}
\label{iv.8}
  \|\varphi_n\|_{\mathcal{C}_T^{2+\t,1+\frac{\t}{2}}(\bar \Omega\times \R;\R)}=1,
\end{equation}
where $\mathcal{C}_T^{2+\t,1+\frac{\t}{2}}(\bar \Omega\times \R;\R)$ denotes the set of functions
$u\in \mathcal{C}^{2+\t,1+\frac{\t}{2}}(\bar \Omega\times \R;\R)$ that are $T$-periodic in time.
Note that, since $\bar{\O}\subset \bar{\mathcal{O}}_n$, $\varphi_n\in E$ for all $n\geq n_0$. By the
compactness of the imbedding
$$
   \mathcal{C}_T^{2+\t,1+\frac{\t}{2}}(\bar \Omega\times \R;\R) \hookrightarrow
   \mathcal{C}_T^{2,1}(\bar \Omega\times \R;\R),
$$
we can extract a subsequence of $\v_n$, relabeled by $n\geq n_0$, such that
\begin{equation}
\label{iv.9}
\lim_{n\to \infty} \varphi_n=:\varphi\quad \hbox{in}\;\; \mathcal{C}_T^{2,1}(\bar\O\times \R;\R).
\end{equation}
From the definition of the $\v_n$'s, it follows from \eqref{iv.9} that
\begin{equation*}
  \left\{ \begin{array}{ll}  \mathcal{P}\varphi  = \sigma_\infty \varphi & \quad \hbox{in}\;\; Q_T, \\ \mathcal{B}\varphi =0  & \quad \hbox{on}\;\; \p Q_T.\\  \end{array} \right.
\end{equation*}
Moreover, $\varphi\gneq 0$ in $Q_T$. Therefore, $\varphi$ must be a principal eigenfunction of  $(\mathcal{P},\mathcal{B},Q_T)$. By the uniqueness of the principal eigenvalue, this entails
\eqref{iv.7} and ends the proof.
\end{proof}

\section{Proof of Theorem \ref{th11}}

\noindent First, we will show that \eqref{1.7} is necessary for the existence of a positive solution of \eqref{1.1}. Indeed, the fact that $\Sigma(\l)<0$ has been already established by Theorem \ref{th22}.
 Let $u$ be a positive solution of \eqref{1.1}. Then, also by Theorem \ref{th22},
$u\gg 0$ solves
\begin{equation*}
  \left\{ \begin{array}{ll} \left[\mathcal{P}- \l m(x,t) +a(x,t) f(u)\right]u=0 & \quad \hbox{in}\;\; Q_T, \\[1ex] \mc{B}u =0  & \quad \hbox{on}\;\; \p Q_T,\\  \end{array} \right.
\end{equation*}
and hence, by the uniqueness of the principal eigenvalue,
\begin{equation*}
\sigma[\mathcal{P}-\lambda m+af(u), \mc{B}, Q_T]=0.
\end{equation*}
By the continuity of $f$ and $u$, there exists a constant $\g>0$ such that
\begin{equation*}
f(u(x,t))<\gamma \quad \hbox{for all} \;\;(x,t)\in \bar{Q}_T.
\end{equation*}
Consequently, thanks to Proposition \ref{pr21}(a),  we can infer that
\begin{align*}
0 & =\sigma[\mathcal{P}-\lambda m+af(u), \mc{B}, Q_T]\\ & <\sigma[\mathcal{P}-\lambda m+\gamma a, \mc{B}, Q_T]=\Sigma(\lambda,\gamma)< \Sigma(\lambda, \infty)
\end{align*}
because $a\gneq 0$ in $Q_T$. This shows the necessity of \eqref{1.7}.
\par
Subsequently, we will assume \eqref{1.7}, i.e.,
\begin{equation}
\label{v.1}
\Sigma(\lambda)<0<\Sigma(\lambda,\infty).
\end{equation}
Let $\v\gg 0$ denote the principal eigenfunction associated to $\Sigma(\lambda)$ normalized so that
$\|\varphi\|_{\infty}=1$. The fact that $\Sigma(\lambda)<0$ entails that $\underline u :=\e \v$
is a positive subsolution of \eqref{1.1} for sufficiently small $\e>0$. Indeed, by ($H_f$), we find from \eqref{v.1} that
\begin{equation*}
0<f(\varepsilon)< -\Sigma(\lambda)/\|a\|_{\infty}
\end{equation*}
for sufficiently small $\e>0$. Thus, $\underline{u}:=\varepsilon\v$  satisfies
\begin{equation*}
(\mathcal{P}-\lambda m+af(\underline{u}))\underline{u}=(\Sigma(\lambda)+af(\varepsilon \varphi))\underline{u}\leq (\Sigma(\lambda)+\|a\|_\infty f(\varepsilon))\underline{u}<0 \quad \hbox{in}\;\; Q_T.
\end{equation*}
Moreover,
\begin{equation*}
\mc{B}\,\underline{u}=\varepsilon \mc{B}\,\varphi=0 \quad \hbox{on} \;\; \partial Q_T.
\end{equation*}
Hence, $\underline{u}$ is a positive strict subsolution of \eqref{1.1} for sufficiently small $\e>0$.
\par
To establish the existence of arbitrarily large supersolutions we proceed as follows. Consider, for sufficiently large $n \in \N$,
\begin{equation*}
\mathcal{O}_n:=\O \cup \left\{x \in \R^N\;\; : \;\; \hbox{dist}(x,\Gamma_0)<\tfrac{1}{n}\right\}.
\end{equation*}
For sufficiently large $n$, say $n\geq n_0$, $\G_1$ is a common part of $\p\O$ and of $\p\mc{O}_n$.
Subsequently, we consider
\begin{equation}
\label{v.2}
\Gamma_{0,n}:=\partial \mathcal{O}_n\setminus \Gamma_1.
\end{equation}
By construction,
$$
  \hbox{dist}(\Gamma_{0,n},\G_0)=\tfrac{1}{n} \quad \hbox{for all}\;\; n\geq n_0.
$$
Now, much like in Section 4, for every $n\geq n_0$, we consider $\mf{L}_n$ and  $\mc{B}_n$ satisfying
\eqref{iv.3} and \eqref{iv.4}, as well as the parabolic operator $\mc{P}_n:=\p_t+\mf{L}_n$. Then, by Theorem \ref{th41}, we have that
\begin{equation}
\label{v.3}
\lim_{n\to \infty}\sigma[\mathcal{P}_n-\lambda m_n+\gamma a_n, \mc{B}_n, \mathcal{O}_{n}\times (0,T)] =\Sigma(\lambda,\gamma),
\end{equation}
where $m_n$ and $a_n$ are continuous periodic extensions of $m$ and $a$ to $Q_{T,n}$.
\par
Since $\Sigma(\lambda,\infty)>0$, by Theorem \ref{th31}, we have that $\Sigma(\lambda,\gamma)>0$ for sufficiently large $\g$, say $\g >\g_0>0$. Subsequently, we fix one of those $\gamma$'s. According to \eqref{v.3}, we can enlarge $n_0$ so that
\begin{equation}
\label{v.4}
\Sigma_n(\l,\g)\equiv \sigma[\mathcal{P}_n-\lambda m_n+\gamma a_n, \mc{B}_n, \mathcal{O}_{n}\times (0,T)]>0 \quad
\hbox{for all} \;\; n\geq n_0.
\end{equation}
Pick $n\geq n_0$ and let $\psi_n\gg 0$ be any positive eigenfunction associated to $\Sigma_n(\l,\g)$. We claim that  $\overline{u}:=\kappa \psi_n$ is a positive strict supersolution of \eqref{1.1} for sufficiently large $\kappa>1$.  Note that
\begin{equation}
\label{v.5}
\psi_n(x,t)>0 \quad \hbox{for all} \;\; (x,t)\in \bar{Q}_T.
\end{equation}
Indeed, $\psi_n(x,t)>0$ for all $(x,t)\in \mc{O}_n\times [0,T]$. Thus, $\psi_n(x,t)>0$
for all $x\in \O\cup\G_0$ and $t\in [0,T]$. Consequently, since $\psi_n(x,t)>0$ for all $x\in \G_1$ and
$t\in [0,T]$, \eqref{v.5} holds true. Thanks to \eqref{v.5}, we have that, for every $\kappa>0$,
\begin{equation*}
\mc{B}\, \overline{u}(x,t)=\overline{u}(x,t)>0 \quad  \hbox{for all}\;\; (x,t) \in \Gamma_0\times [0,T].
\end{equation*}
Moreover, by construction, $\mc{B}\, \overline{u}=0$ on $\Gamma_1\times [0,T]$. Therefore, $\mc{B}\,\overline{u}\geq 0$ on $\p\Omega\times [0,T]$.
\par
On the other hand, by \eqref{v.4} and \eqref{v.5}, we have that
\begin{equation}
\label{v.6}
(\mathcal{P}_n-\lambda m_n)\psi_n=\Sigma_n(\l,\g) \psi_n- \gamma a_n \psi_n > -\gamma a_n \psi_n \quad \hbox{in} \;\; Q_T.
\end{equation}
By construction,  $\mc{P}_n=\mc{P}$, $a_n=a$ and $m_n=m$ in $Q_T$. Moreover,
by \eqref{v.5}, there exists a constant $\mu>0$ such that $\psi_n>\mu$ in $\bar{Q}_T$.
Thus, owing to ($H_f$), it follows from \eqref{v.6} that
\begin{equation}
\label{v.7}
(\mathcal{P}-\lambda m+af(\overline{u}))\overline{u}> a(f(\kappa \psi_n)-\gamma)\overline{u}\geq a(f(\kappa\mu)-\gamma)\overline{u}>0 \quad \hbox{in} \;\; Q_T
\end{equation}
for sufficiently large $\kappa>0$. Therefore, $\overline{u}$ is a positive strict supersolution of \eqref{1.1} for sufficiently large $\kappa>0$.
\par
Finally, either shortening $\varepsilon>0$, or enlarging $\kappa>0$, if necessary, one can assume that
\begin{equation*}
0\ll \underline{u}=\e\v  \leq \varepsilon \leq \kappa \mu  \leq  \overline{u}=\k \psi_n.
\end{equation*}
Consequently, by adapting the argument of Amann \cite{Amann} to a periodic-parabolic context (see, e.g., Hess \cite{Hess}, or Daners and Koch-Medina \cite{DK}), it becomes apparent that \eqref{1.1} admits a positive solution. This shows the existence of a positive solution for \eqref{1.1}. The uniqueness is already known by Theorem \ref{th22}.

\section{Proof of Theorem \ref{th14}}

\subsection{Preliminaries}

\noindent Throughout this section,  according to \cite[Def. 3.2]{LG20}, for any given connected open subset $G$ of $\R^N\times \R$, a point $(x,t)\in\p G$ is said to belong to the \emph{flat top boundary} of $G$, denoted by $\p_{FT}G$, if there exists $\e>0$ such that
\begin{itemize}
\item $B_\e(x)\times \{t\}\subset \p G$, where $B_\e(x):=\{y\in\R^N:|y-x|<\e\}$.
\item $B_\e(x)\times (t-\e,t)\subset G$.
\end{itemize}
Then, the \emph{lateral boundary} of $G$, denoted by $\p_{L}G$, is defined by
$$
  \p_{L}G =\partial G\setminus \p_{FT} G.
$$
The next generalized version of the parabolic weak maximum principle holds; it is \cite[Th. 3.4]{LG20}.

\begin{theorem}
\label{th61}
Let $G$ be a connected open subset of $\R^N\times \R$ whose flat top boundary $\p_{FT}G$ is nonempty and consists of finitely many components with disjoint closures, and suppose $c=0$, i.e.,
$\mf{L}$ does not have zero order terms. Let $u \in \mc{C}(\bar G)\cap
\mc{C}^2(G)$ be such that
$$
 \mc{P}u= \p_t u + \mf{L}u\leq 0 \quad \hbox{in}\;\; G.
$$
Then,
$$
  \max_{\bar G}u = \max_{\p_{L}G}u.
$$
\end{theorem}

Subsequently, for every $\e\geq 0$, we set
$$
  [a\geq \e]:=\{(x,t)\in\bar Q_T\;:\; a(x,t)\geq \e\}=a^{-1}([\e,\infty)),
$$
and denote by $\v_\g\gg 0$ the unique principal eigenfunction associated to
$\Sigma(\l,\g)$ such that $\|\v_\g\|_{L^\infty(Q_T)}=1$. Then, by \cite[Th. 5.1]{LG20}, the next result holds.

\begin{theorem}
\label{th62}
Assume ($H_\Omega$), ($H_\mf{L}$), ($H_\mc{B}$), ($H_{a}$),
\begin{equation}
\label{6.1}
  a(x,t)>0 \quad \hbox{for all}\;\; (x,t)\in\p\O\times [0,T],
\end{equation}
and $\Sigma(\l,\infty)<\infty$ for some $\l\in\R$. Then,
\begin{equation}
\label{6.2}
\lim_{\g\ua \infty}\v_\g =0
\end{equation}
uniformly on compact subsets of $(\O\times [0,T])\cap[a>0]$.
\end{theorem}

\subsection{Proof of Theorem \ref{th13}}

Note that, besides ($H_\Omega$), ($H_\mf{L}$), ($H_\mc{B}$), ($H_{a}$), we are assuming \eqref{6.1}, or, equivalently, \eqref{1.9}. Our proof will proceed by contradiction. So, assume that $\Sigma(\tilde \l,\infty)<\infty$ for some $\tilde \l\in\R$. Then, by Theorem \ref{th31},
\begin{equation}
\label{6.3}
  \Sigma(\l,\infty)<\infty \quad \hbox{for all}\;\;\l\in\R.
\end{equation}
Fix $\l\in\R$, and assume that $c-\l m\geq 0$ in $Q_T$. Once completed the proof in this case, we will show how to conclude the proof when $c-\l m$ changes of sign.
\par
For sufficiently small $\e>0$, let $G_\e$ be an open set such that
$$
  \p_{L} G_\e \subset [a>0] \quad \hbox{and}\quad t = t_Q-\e \quad \hbox{if}\;\;
  (x,t)\in \p_{FT} G_\e,
$$
as illustrated in the first picture of Figure \ref{Fig6}. In this figure, as in all subsequent pictures
of this paper, we are representing the nodal behavior of the weight function $a(x,t)$ in the parabolic cylinder $Q_{\tau}=\O\times [0,\tau]$ for some $\tau \in (t_P,T)$.
\par

Since $\p_{L} G_\e$ is a compact subset of the open set $[a>0]$,  Theorem \ref{th62} implies that
\begin{equation}
\label{6.4}
  \lim_{\g\ua \infty}\v_\g =0 \quad \hbox{uniformly on}\;\; \p_{L} G_\e.
\end{equation}
Moreover, by the definition of $\v_\g$, we have that, for every $\g\geq 0$,
$$
  \p_t \v_\g +(\mf{L}-c)\v_\g =(-c+\l m-\g a +\Sigma(\l,\g))\v_\g \leq \Sigma(\l,\g)\v_\g,
$$
because $-c+\l m\leq 0$ and $a\geq 0$. Thus, the auxiliary function
\begin{equation}
\label{6.5}
  \psi_\g:= e^{-t \Sigma(\l,\g)} \v_\g
\end{equation}
satisfies
\begin{align*}
\p_t \psi_\g +(\mf{L}-c)\psi_\g & = -\Sigma(\l,\g) e^{-t \Sigma(\l,\g)} \v_\g  +
e^{-t \Sigma(\l,\g)} \p_t \v_\g+ e^{-t \Sigma(\l,\g)} (\mf{L}-c) \v_\g  \\ & =
e^{-t \Sigma(\l,\g)} \left[- \Sigma(\l,\g)\v_\g  + \p_t \v_\g +(\mf{L}-c) \v_\g\right] \leq 0.
\end{align*}
Hence, it follows from Theorem \ref{th61} that, for every $\g\geq 0$,
\begin{equation}
\label{6.6}
  \max_{\bar G_\e}\psi_\g = \max_{\p_{L} G_\e}\psi_\g.
\end{equation}
On the other hand, by \eqref{6.3},
\begin{equation}
\label{vi.7}
   \Sigma(\l,\g)\leq \Sigma(\l,\infty)<\infty.
\end{equation}
Thus, \eqref{6.4} and \eqref{6.5} imply that
\begin{equation*}
    \lim_{\g\ua \infty}\psi_\g =0 \quad \hbox{uniformly on}\;\; \p_{L} G_\e.
\end{equation*}
Consequently, according to \eqref{6.6}, we actually have that
\begin{equation*}
    \lim_{\g\ua \infty}\psi_\g =0 \quad \hbox{uniformly in}\;\; \bar G_\e.
\end{equation*}
Equivalently, by \eqref{6.5} and \eqref{vi.7},
\begin{equation*}
    \lim_{\g\ua \infty}\v_\g =0 \quad \hbox{uniformly in}\;\; \bar G_\e.
\end{equation*}
As this holds regardless the size of $\e>0$, we find that
\begin{equation}
\label{6.8}
  \lim_{\g\ua\infty}\v_\g = 0 \quad \hbox{in}\;\; G_0:=\cup_{\e>0} G_\e,
\end{equation}
uniformly in $\bar G_\e$ for all $\e>0$.

\begin{figure}[h!]
\centering
\includegraphics[scale=1]{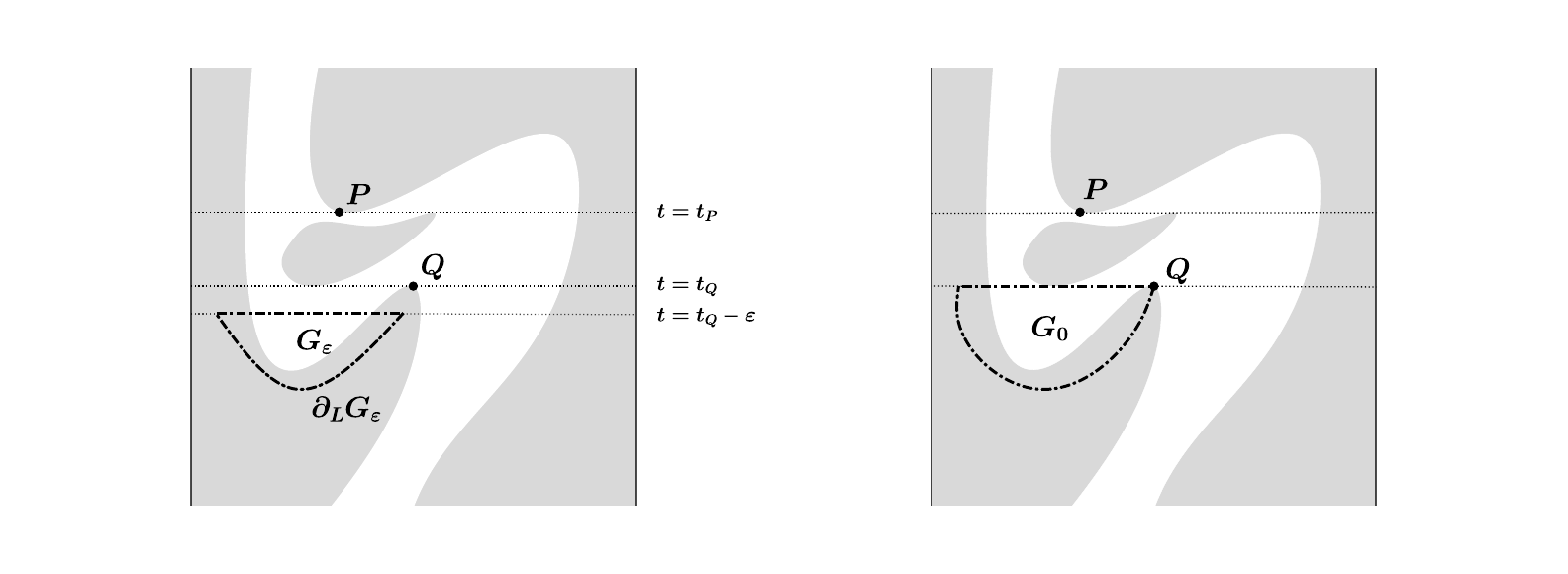}
\caption{The construction of the open set $G_0$. }
\label{Fig6}
\end{figure}

To show the validity of \eqref{6.8} at the level $t=t_Q$, one needs an additional argument. Consider any closed ball, $\bar B$, such that
$$
  \bar B\subset \mathrm{int\,}a^{-1}(0),\quad B\cap [t>t_Q]\neq \emptyset,\quad
  B\cap [t<t_Q]\neq \emptyset,
$$
as illustrated in the first picture of Figure \ref{Fig7}, where the ball has been centered at a
point $R=(x_R,t_R)$ with $t_R=t_Q$. Since $\Sigma(\l,\infty)<\infty$, in the interior of $a^{-1}(0)$ one can combine the theorem of Eberlein--Schmulian (see, e.g., \cite[Th. 3.8]{LG13}) with Theorem 4 of Aronson and Serrin \cite{ArSe} to show the existence of a sequence $\{\g_n\}_{n\geq 1}$ such that
$$
  \lim_{n\to \infty}\g_n=\infty\quad\hbox{and}\quad \lim_{n\to\infty} \v_{\g_n}=\v_\infty
  \quad \hbox{uniformly in}\;\;\bar B
$$
for some smooth function $\v_\infty \in \mc{C}^{2,1}(\bar B)$. Since
\begin{equation*}
 \bar B\cap [t<t_Q]\subset G_0,
\end{equation*}
we can infer from \eqref{6.8} that $\v_\infty(x,t)=0$ if $(x,t)\in \bar B$ with $t\leq t_Q$. Thus, there is a region $G$, like the one represented in the second picture of Figure \ref{Fig7}, such that
\begin{equation}
\label{6.9}
  \lim_{\g\ua \infty}\v_{\g}=0\quad \hbox{uniformly on}\;\;\p_{L} G.
\end{equation}

\par
\begin{figure}[h!]
\centering
\includegraphics[scale=1.3]{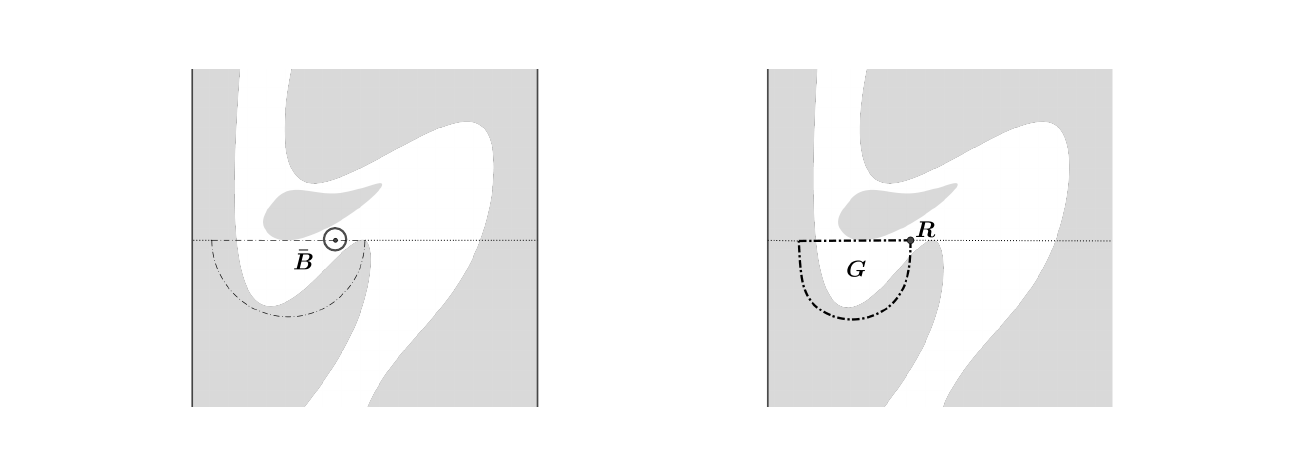}
\caption{The construction of the open set $G$. }
\label{Fig7}
\end{figure}

Consequently, applying again the Theorem \ref{th62} to $\psi_\g$, we are also driven to
\begin{equation}
\label{6.10}
   \lim_{\g\ua \infty}\v_{\g}=0\quad \hbox{uniformly on}\;\;\bar G.
\end{equation}
Next, for sufficiently small $\e>0$, we consider any open set $H_\e$ such that
$$
  t=t_P-\e\quad \hbox{if}\;\; (x,t)\in \p_{FT} H_\e,
$$
\begin{equation}
\label{6.11}
  \bar H_\e \cap [t\leq t_Q] \subset \bar G,
\end{equation}
and
\begin{equation}
\label{6.12}
   \p_{L} H_\e \cap (t_Q,t_P-\e]\subset [a>0]
\end{equation}
as illustrated by the first picture of Figure \ref{Fig8}, where $S=(x_S,t_S)$ stands for any point in $M$ with $t_S=t_Q$; $H_\e$ exists by our structural assumptions. According to \eqref{6.10} and
\eqref{6.11}, we have that
\begin{equation}
\label{6.13}
   \lim_{\g\ua \infty}\v_{\g}=0\quad \hbox{uniformly on}\;\;\bar H_\e \cap[t\leq t_S].
\end{equation}
Moreover, by Theorem \ref{th62}, it follows from \eqref{6.12} that,
for arbitrarily small $\eta>0$,
\begin{equation}
\label{6.14}
    \lim_{\g\ua \infty}\v_{\g}=0\quad \hbox{uniformly on}\;\;  \p_{L} H_\e  \cap [t_S+\eta,t_P-\e].
\end{equation}
Note that, thanks to \eqref{6.13} and \eqref{6.14}, we have that
\begin{equation}
\label{6.15}
  \lim_{\g\ua\infty}\v_\g =0 \quad \hbox{point-wise on} \;\; \p_{L} H_\e
\end{equation}
and uniformly on $\p_{L} H_\e \setminus B_\d(S)$ for sufficiently small $\d>0$, where $B_\d(S)$
stands for the open ball centered at $S$ with radio $\d>0$. Naturally, by our assumptions, the function $\psi_\g$ defined in  \eqref{6.5} also approximates
zero point-wise on  $\p_{L} H_\e$ and uniformly on $\p_\mathrm{L} H_\e \setminus B_\d(S)$
for sufficiently small $\d>0$. Moreover, since
$$
   \p_t \psi_\g +(\mf{L}-c)\psi_\g \leq 0\quad \hbox{in}\;\; Q_T,
$$
$\psi_\g$ provides us with a positive subsolution of the parabolic operator
$\p_t+\mf{L}-c$. Thus, by the parabolic Harnack inequality, there exists a constant $C>0$ such that
\begin{equation}
\label{vi.16}
  \max_{\p_{L} H_\e}\psi_\g \leq C \min_{\p_{L} H_\e}\psi_\g.
\end{equation}
Hence,
$$
  \lim_{\g\ua \infty} \max_{\p_{L} H_\e}\psi_\g=0
$$
and therefore,
$$
  \lim_{\g\ua \infty} \max_{\p_{L} H_\e}\v_\g =\lim_{\g\ua \infty} \max_{\p_{L} H_\e}\psi_\g=0.
$$
Consequently, once again by Theorem \ref{th61}, one can infer that
$$
  \lim_{\g\ua \infty} \v_\g =0 \quad \hbox{uniformly on}\;\; \bar H_\e
$$
for sufficiently small $\e>0$.
\par
Now, consider the open set $H_0$ defined through
$$
  H_0:=\cup_{\e>0} H_\e,
$$
as sketched in the second picture of Figure \ref{Fig8}.

\begin{figure}[h!]
\centering
\includegraphics[scale=1]{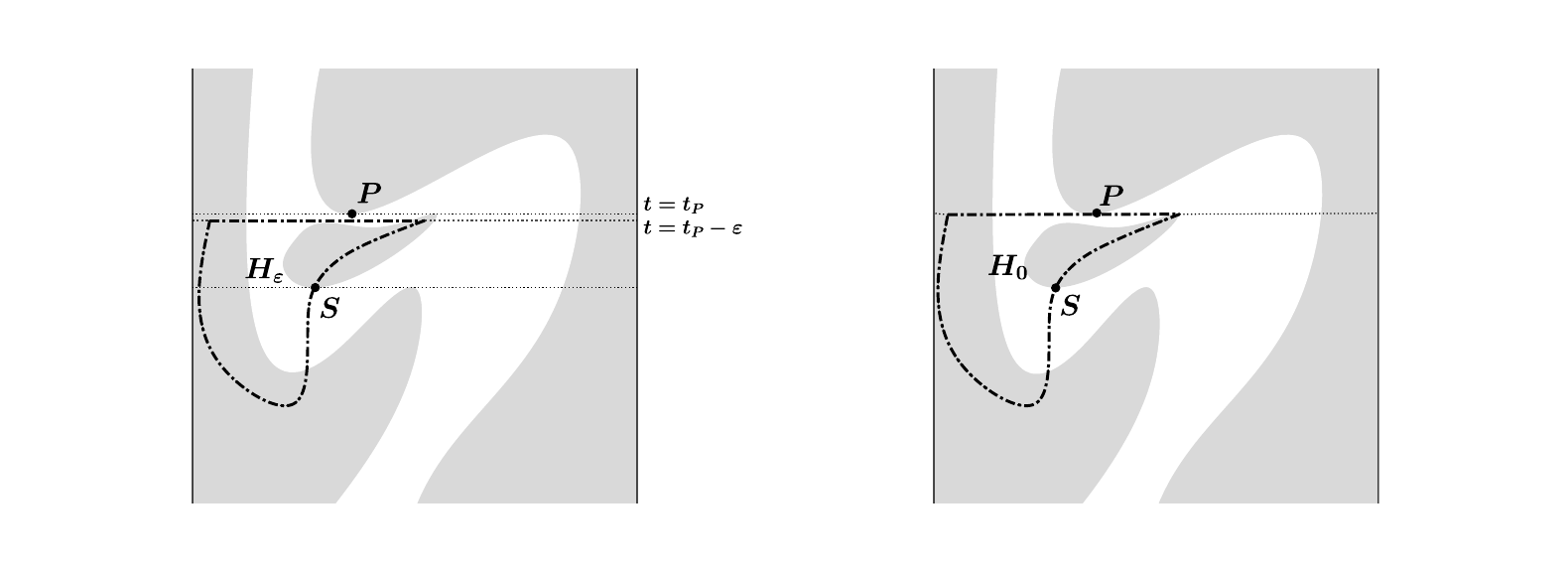}
\caption{The construction of the open set $H_0$. }
\label{Fig8}
\end{figure}

Proceeding with $H_0$ as we previously did with
$G_0$,  it becomes apparent that one can construct an open subset $H\subset H_0$, like the one sketched in the first picture of Figure \ref{Fig9}, such that
\begin{equation*}
  \lim_{\g\ua \infty}\v_{\g}=0\quad \hbox{uniformly on}\;\;\p_{L} H.
\end{equation*}
Therefore, by Theorem \ref{th61},
\begin{equation}
\label{vi.17}
  \lim_{\g\ua \infty}\v_{\g}=0\quad \hbox{uniformly on}\;\;\bar H.
\end{equation}

\begin{figure}[h!]
\centering
\includegraphics[scale=1]{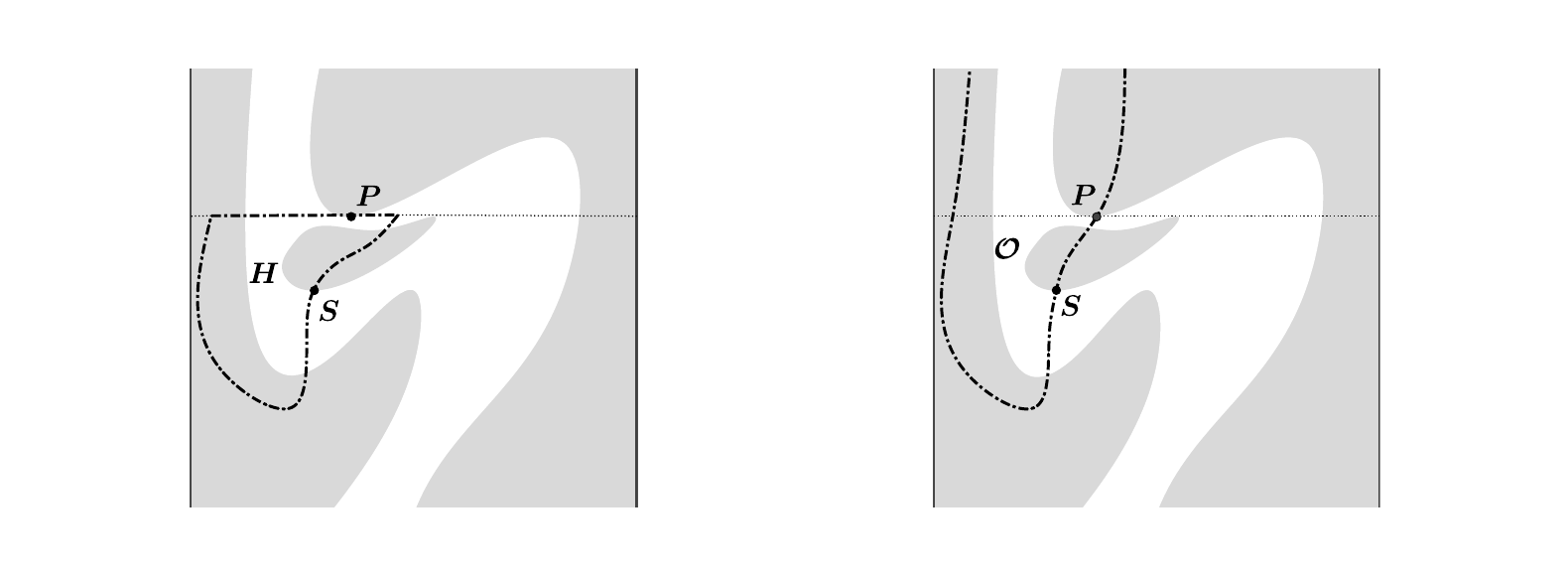}
\caption{The construction of the open set $\mc{O}$. }
\label{Fig9}
\end{figure}

Finally, considering an open subset $\mc{O}$ of $Q_T$ of the type described in the second picture
of Figure \ref{Fig9}, with
$$
  \bar{\mc{O}}\cap [t\leq t_P]\subset \bar H\quad \hbox{and}\quad
  \p_{L}\mc{O}\cap [t_P<t\leq T]\subset [a>0],
$$
at the light of the analysis already done in this proof it is apparent that
$$
   \lim_{\g\ua\infty}\v_\g =0\quad \hbox{uniformly in}\;\; \bar{\mc{O}}.
$$
By Theorem \ref{th62}, this implies that
$$
 \lim_{\g\ua\infty}\v_\g(x,T) =0\quad \hbox{uniformly in}\;\; \bar{\O}.
$$
Thus, since $\v(x,t)$ is $T$-periodic,
$$
 \lim_{\g\ua\infty}\v_\g(x,0) =0\quad \hbox{uniformly in}\;\; \bar{\O},
$$
which implies
$$
  \lim_{\g\ua \infty}\v_\g =0 \quad \hbox{uniformly on}\;\;\bar Q_T.
$$
This is impossible, because $\|\v_\g\|_{L^\infty(Q_T)}=1$. This contradiction shows that
\begin{equation}
\label{vi.18}
  \Sigma(\l,\infty)=\infty.
\end{equation}
In the general case when $c-\l m $ changes of sign, we can pick a sufficiently large $\o>0$ such that
$$
  c-\l m +\o \geq 0 \qquad \hbox{in}\;\;\bar{Q}_T.
$$
Then, by the result that we have just proven,
$$
   \lim_{\g\ua \infty} \sigma[\mathcal{P}-\lambda m+\o+\gamma a, \mc{B}, Q_T]=
   \Sigma(\l,\infty)+\o =\infty.
$$
Therefore, \eqref{vi.18} also holds in this case. This ends the proof of Theorem \ref{th13}.

\subsection{Proof of Theorem \ref{th14}} Assume that $n=2$ and that $\Sigma(\tilde\l,\infty)<\infty$
for some $\tilde \l\in\R$. Then, by Theorem \ref{th31}, \eqref{6.3} holds. Arguing as in the
proof of Theorem \ref{th13}, we fix $\l\in\R$ and assume $c-\l m\geq 0$ in $Q_T$. The general case
follows with the final argument of the proof of Theorem \ref{th13}.
\par
Repeating the argument of the proof of Theorem \ref{th13}, one can construct an open subset, $H$, of $Q_T$, with $H\subset [t\leq t_{P_2}]$, as illustrated in the first picture of Figure \ref{Fig10}, satisfying \eqref{vi.17}.

\begin{figure}[h!]
\centering
\includegraphics[scale=1]{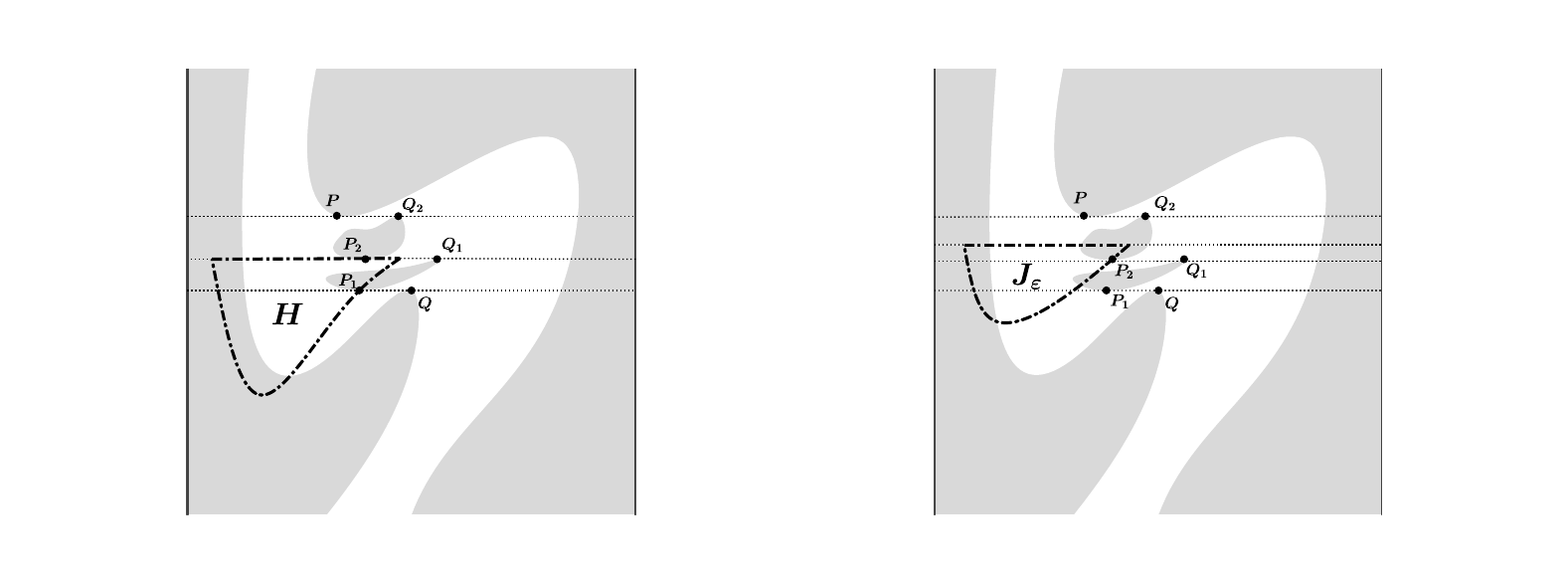}
\caption{The construction of the open sets $J_\e$, $\e>0$. }
\label{Fig10}
\end{figure}

Arguing as in the proof of Theorem \ref{th13}, for sufficiently small $\e>0$, there exists an open set
$J_\e$ such that
\begin{equation}
\label{vi.19}
t=t_P-\e=t_{Q_2}-\e\quad \hbox{if}\;\; (x,t)\in \p_{FT} J_\e,\qquad
  \bar J_\e \cap [t\leq t_{P_2}] \subset \bar H,
\end{equation}
and
\begin{equation}
\label{vi.20}
   \p_{L} J_\e \cap (t_{P_2},t_P-\e]\subset [a>0]
\end{equation}
like the one shown in the second picture of Figure \ref{Fig10}; $J_\e$ exists by our structural assumptions. By \eqref{vi.17}, \eqref{vi.19} and \eqref{vi.20}, we find that
\begin{equation}
\label{6.20}
   \lim_{\g\ua \infty}\v_{\g}=0\quad \hbox{uniformly on}\;\;\bar J_\e \cap[t\leq t_{P_2}].
\end{equation}
Moreover, as in the proof of Theorem \ref{th13}, it is apparent that
\begin{equation}
\label{6.21}
  \lim_{\g\ua\infty}\v_\g =0 \quad \hbox{point-wise on} \;\; \p_{L} J_\e
\end{equation}
and uniformly on $\p_{L} J_\e \setminus B_\d(P_2)$ for sufficiently small $\d>0$. By our assumptions, the function $\psi_\g$ defined in \eqref{6.5} also approximates
zero point-wise on  $\p_{L} J_\e$ and uniformly on $\p_{L} J_\e \setminus B_\d(P_2)$
for arbitrarily small $\d>0$, and
$$
   \p_t \psi_\g +(\mf{L}-c)\psi_\g \leq 0\quad \hbox{in}\;\; Q_T.
$$
Therefore, arguing as in the proof of Theorem \ref{th13}, it follows from the parabolic Harnack inequality that
$$
  \lim_{\g\ua \infty} \v_\g =0 \quad \hbox{uniformly on}\;\; \bar J_\e.
$$
Next, we consider the open set $J_0$ defined by
$$
  J_0:=\cup_{\e>0} J_\e,
$$
which has been represented in the first picture of Figure \ref{Fig11}. Arguing with $J_0$ as we did with
$G_0$ in the proof of Theorem \ref{th13},  it becomes apparent that there exists an open subset $J\subset J_0$,  such that
\begin{equation*}
  \lim_{\g\ua \infty}\v_{\g}=0\quad \hbox{uniformly on}\;\;\p_{L} J
\end{equation*}
and $J=\mathcal{O}\cap[t\leq t_P]$ (see the right picture of Figure \ref{Fig11}), where
$\mc{O}$ is a prolongation of $J$ up to $t=T$ with
$$
  \p_{L}(\mc{O}\setminus J)\subset [a>0].
$$

\begin{figure}[h!]
\centering
\includegraphics[scale=1]{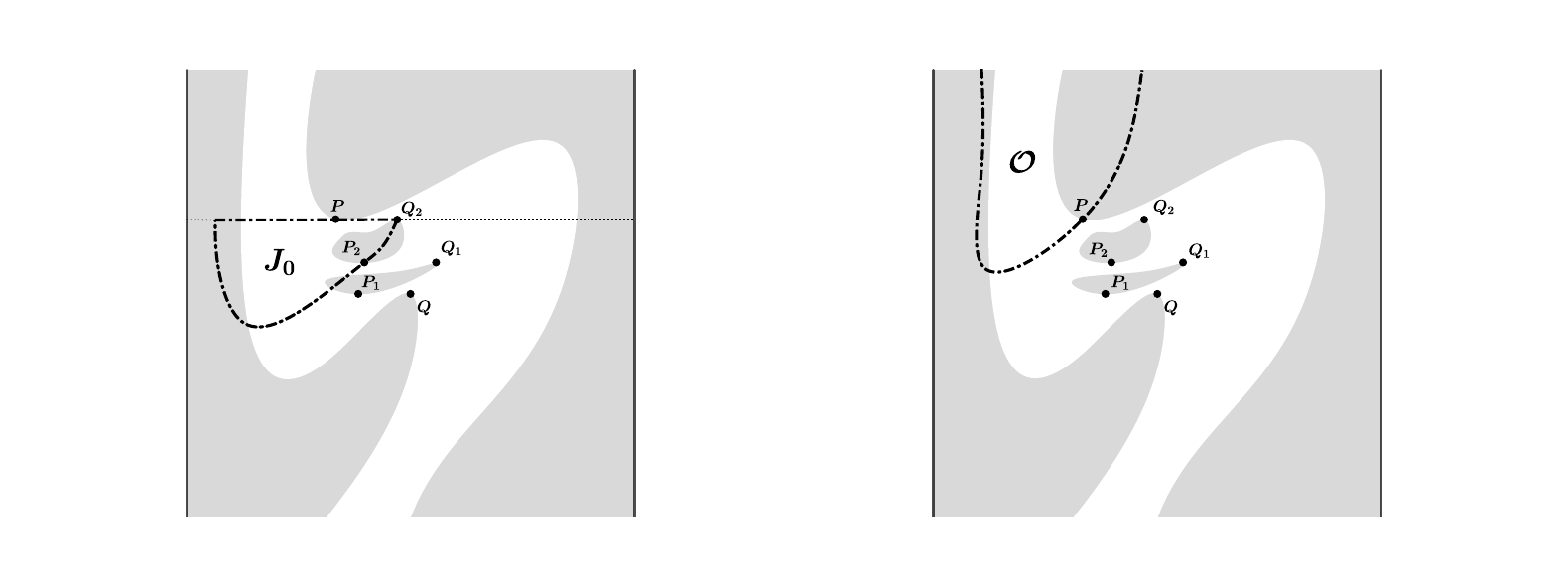}
\caption{The construction of $J_0$, $J$ and $\mc{O}$. }
\label{Fig11}
\end{figure}

Repeating the last steps of the proof of Theorem \ref{th13}, the proof of Theorem
\ref{th14} for $n=2$ follows. The proof in the general case $n\geq 2$ follows by induction,
but omit its technical details here by repetitive.

\end{document}